\documentclass[11pt]{article}
\usepackage{amsmath,amsfonts,amsthm,amscd,amssymb,graphicx}

\usepackage{graphicx}

\usepackage{amssymb}
\usepackage[colorlinks,citecolor=blue,linkcolor=blue]{hyperref}
\usepackage{xcolor}

\def\rit{{\Bbb R}}
\def\cit{{\Bbb C}}

\def\zit{{\Bbb Z}}

\def\eps{\varepsilon}

\newcommand{\per}{{per}}

\newtheorem{theorem}{Theorem}[section]
\newtheorem{lemma}[theorem]{Lemma}
\newtheorem{e-proposition}[theorem]{Proposition}

\newtheorem{e-definition}[theorem]{Definition\rm}

\newtheorem{theoreme}{Th\'eor\`eme}[section]

\newtheorem{proposition}[theoreme]{Proposition}

\def\og{\leavevmode\raise.3ex\hbox{$\scriptscriptstyle\langle\!\langle$~}}
\def\fg{\leavevmode\raise.3ex\hbox{~$\!\scriptscriptstyle\,\rangle\!\rangle$}}

\def\beq{\begin{equation}}
\def\eeq{\end{equation}}

\begin{document}



%

\centerline{\Large \bf Bifurcations of viscous shear flows in a strip}

\bigskip

\centerline{D. Bian\footnote{School of Mathematics and Statistics, Beijing Institute of Technology, Beijing 100081, China. Email: biandongfen@bit.edu.cn and emmanuelgrenier@bit.edu.cn}, 
E. Grenier$^1$,
M. Haragus\footnote{FEMTO-ST Institute, University of Franche Comt\'e,  25030 Besan\c con cedex, France. Email: mharagus@univ-fcomte.fr}
}



\subsubsection*{Abstract}


It is well-established that shear flows in a periodic strip are linearly unstable for the incompressible Navier Stokes equations
provided the viscosity is small enough.
In this article, under a natural spectral assumption which is satisfied for convex or concave analytic flows, 
we prove that shear flows undergo a Hopf bifurcation near their upper marginal stability curve. 
In particular, near this curve, there exist solutions which are periodic in $t$ and $x$.


\section{Introduction}


In this paper, we consider the incompressible Navier-Stokes equations in the strip $\Omega = \rit \times (0,1)$
 \beq \label{NS1} 
\partial_t u^\nu + (u^\nu \cdot \nabla) u^\nu - \nu \Delta u^\nu + \nabla p^\nu = f^\nu,
\eeq
\beq \label{NS2}
\nabla \cdot u^\nu = 0 ,
\eeq
together with the Dirichlet boundary condition
\beq \label{NS3} 
u^\nu = 0 \quad \hbox{when} \quad y = 0 \quad \hbox{and} \quad y = 1.
\eeq
We are interested in the nonlinear evolution of small perturbations of a given shear flow
$$
U(y) = (U_s(y),0).
$$
Note that this shear flow is a stationary solution of Navier-Stokes equations
provided we add the forcing term 
$$
f^\nu = (- \nu \Delta U_s,0).
$$ 
We assume that $U_s(y)$ is a smooth function, symmetric with respect to $y = 1/2$, that $U_s(0) = 0$ 
and that $\partial_y U_s(0) \ne 0$. In all this paper, the  various functions will be symmetric with respect to $y = 1/2$.

\medskip

We can distinguish between two kinds of shear flows. The first possibility is that the shear flow is spectrally unstable for the linearized Euler equations. In this case,
it is  nonlinearly unstable for Navier-Stokes equations: as shown in \cite{GN2}, small and smooth perturbations may
grow and reach a magnitude $O(1)$ in $L^\infty$.

The second possibility is that the shear flow is spectrally stable for the linearized Euler equations. However, it is
well-established in physics \cite{Reid,Reid2} and  proved in \cite{Guo2}, that such shear flows are spectrally
unstable for Navier-Stokes equations provided the viscosity is small enough. According to Rayleigh's criterium,
concave or convex flows belong to this second class. 

\medskip

In both cases, it is physically expected that generic small perturbations increase and ultimately lead to fully turbulent flows.
According to \cite{Schmidt}, the development of instabilities may follow various scenarios, 
like for instance the ``spectral'' scenario or the ``by-pass'' scenario.

\medskip

In the ``spectral'' scenario, the unavoidable noise of any experiment triggers some linear instability, which first increases exponentially
(a fact justified in \cite{Guo2}). Later, nonlinear mechanisms can
no longer be neglected with respect to the linear growth. The linear instability get tamed by the nonlinear terms and saturates. This second phase
gives birth to a new ``equilibrium''
or new ``base state", which in turn becomes unstable. Secondary instabilities appear and grow, leading to the destabilisation of this secondary base state, leading
to a third base state. Instabilities then becomes more and more complex and close, leading to fully developed turbulence
(so called ``route to chaos'').
  
\medskip
  
In this paper, we investigate the second phase of the spectral scenario,
namely the influence of the nonlinear terms on the initial linear instability, both for shear flows of the first kind (which are unstable with
respect to the Euler equation) and of the second kind (which are stable with respect to the Euler equation).
We show that, close to the upper marginal stability curve, there exist time and and space periodic solutions
of the nonlinear Navier-Stokes equations. These solutions are possible ``new base states" which are temporarily 
visited during the transition from laminar to turbulent shear flows.

In this article, we consider the case of a strip $\Omega = \mathbb{R} \times (0,1)$.
The half space case $\Omega = \mathbb{R} \times \mathbb{R}^+$, which requires completely different techniques, is considered in \cite{BGI}.


\section{Spectral analysis and statement of the results}



\subsection{The Orr-Sommerfeld equation}


Let us first introduce the classical Orr-Sommerfeld equations (see \cite{BG4} or \cite{Reid} for more details).
Let $L_{NS}$ be the linearized Navier-Stokes operator near the shear flow profile $U$, defined by
\beq \label{linearNS}
L_{NS} \; v  = (U \cdot \nabla) v + (v\cdot \nabla) U - \nu \Delta v + \nabla q,
\eeq
with $\nabla \cdot v = 0$ and Dirichlet boundary conditions at $y = 0$ and $y = 1$.

The solution of the linear equation
\beq \label{liinearNS2}
\partial_t v + L_{NS} \, v = f
\eeq
is given by Dunford's formula
\beq \label{contour1}
v(t,x) = {1 \over 2 i \pi} \int_\Gamma e^{\lambda t} (L_{NS} + \lambda)^{-1} v_0 \; d\lambda
\eeq
in which $\Gamma$ is a contour on the ``right'' of the spectrum.
This leads to study the resolvent of $L_{NS}$
\beq \label{resolvant}
(L_{NS} + \lambda) v = f,
\eeq
where $f$ is a given forcing term and $\lambda$ a complex number.

We take the Fourier transform in the $x$ variable, with dual variable $\alpha$
and denote by $L_{NS,\alpha}$ the Fourier transform of $L_{NS}$. 
Let $Sp_{\alpha,\nu}$ be the set of the eigenvalues $\lambda$ of $L_{NS,\alpha}$.
Taking advantage of the divergence free condition, we introduce the stream function $\psi$ and take the Fourier
transform in $x$ and the Laplace transform in $t$, which leads to look for solutions of the form
$$
v = \nabla^\perp \Bigl( e^{i \alpha (x - c t) } \psi(y) \Bigr) .
$$
Note that $\lambda = - i \alpha c$, according to the traditional notations. 
We also take the Fourier and Laplace transform of the forcing term $f$ by setting
$$
f = \Bigl( f_1(y),f_2(y) \Bigr) e^{i \alpha (x - c t) } .
$$
We denote by $L_{NS,\alpha}$ the operator $L_{NS}$ after the Fourier transform in $x$.
Taking the curl of (\ref{resolvant}), we get the classical Orr-Sommerfeld equation
\beq \label{Orrmod0}
Orr_{c,\alpha,\nu}(\psi) = \alpha (U_s - c)  (\partial_y^2 - \alpha^2) \psi - \alpha U_s''  \psi  
 + i \nu   (\partial_y^2 - \alpha^2)^2 \psi =  i {\nabla \times f }
\eeq
where
$$
\nabla \times (f_1,f_2) = i \alpha f_2 - \partial_y f_1.
$$
The Dirichlet boundary conditions
gives 
\beq \label{condOrr}
\psi(0) = \partial_y \psi(0) = 0,
\eeq
with similar equalities at $y = 1$,
and the symmetry condition leads to
\beq
\partial_y \psi(1/2) = \partial_y^3 \psi(1/2) = 0.
\eeq
Note that, in order to avoid a singularity at $\alpha = 0$, $Orr_{c,\alpha,\nu}$ is the usual Orr-Sommerfeld operator multiplied by $\alpha$.
When $\alpha = 0$, the Orr-Sommerfeld operator reduces to $- \lambda \partial_y^2 \psi + \nu\partial_y^4 \psi$.
We introduce
$$
\eps = - {\nu \over i \alpha} .
$$
Moreover, if $\psi$ is a solution of (\ref{Orrmod0}) with wavelength $\alpha$ then 
$\bar \psi$ is a solution of the same equation with wavelength $- \alpha$.


\subsection{Shear flows of the first kind}


By definition, the flow is linearly unstable for $\nu = 0$, which means that there exists some $\alpha_0$ and some complex number
$\lambda_0$ with $\Re \lambda_0 > 0$, such that $\lambda_0 \in Sp_{\alpha_0,0}$. 

Using the method of \cite{GN2}, we can prove that, provided $\nu$ remains small enough, there exists $\lambda_\nu$ with
$\Re \lambda_\nu > 0$ such that $\lambda_\nu \in Sp_{\alpha_0,\nu}$.
In particular, for $\nu \le \nu_\star$ with $\nu_\star$ small enough, there exists
$\lambda \in Sp_{\alpha_0,\nu}$ with $\Im \lambda > 0$, namely an ``unstable eigenvalue".

Let us fix such a $\nu$. Then for $\alpha$ large enough, a direct energy estimate shows that all the eigenvalues of  $L_{NS,\alpha}$
have a negative imaginary part. This implies that for every $0 < \nu \le \nu_\star$, there exists $\alpha_+(\nu)$ such that
all the eigenvalues of $L_{NS,\alpha}$ have a negative imaginary part if $\alpha > \alpha_+(\nu)$.
Moreover, if $\alpha < \alpha_+(\nu)$ and is close enough to $\alpha_+(\nu)$, there exists at least one unstable eigenvalue.

To fulfil the bifurcation analysis, we need a little more, namely that for such $\alpha$, there exists only one eigenvalue with a positive real part,
and that this eigenvalue is simple (see assumption (H) below).
The proof of these two last properties is  open.


\subsection{Shear flows of the second kind}


In \cite{Guo2}, we proved that there exist two functions $\alpha_{lower}(\nu)$ and $\alpha_{upper}(\nu)$ such that, as $\nu$ goes to $0$,
$$
\alpha_{lower}(\nu) \sim C_- \nu^{1/3}, \qquad \alpha_{upper}(\nu) \sim C_+ \nu^{1/7}
$$
for some positive constants $C_-$ and $C_+$,
and such that, if $\alpha_{lower}(\nu) < \nu < \alpha_{upper}(\nu)$ 
there exists one and only one eigenvalue $\lambda(\alpha,\nu)$ of $L_{NS,\alpha}$ such that $\Re \lambda(\alpha,\nu) > 0$. 
Moreover, this eigenvalue is simple.

In the case of convex or concave flows, we moreover know \cite{BG4} that there exist no unstable eigenvalue if $\alpha > \alpha_{upper}(\nu)$.

For generic flow, we do not know whether $\alpha_{upper}(\nu)$ satisfies this last property. However, we know that for $\alpha$ large enough,
using a direct energy estimate, $L_{NS,\alpha}$ has no unstable eigenvalue. 
As in the case of shear flows of the first kind, we can construct $\alpha_+(\nu)$.
For $\alpha > \alpha_+(\nu)$, there is no unstable eigenvalue, and for $\alpha < \alpha_+(\nu)$, close to $\alpha_+(\nu)$, $L_{NS,\alpha}$ has
at least one unstable eigenvalue.

As in the previous paragraph, we need to assume that this eigenvalue is unique and simple, which will be stated in the assumption (H) below.


\subsection{A spectral assumption}


We will make the following spectral assumption:

\medskip

{\it \noindent 
Assumption (H):    
There exists a  function $\alpha_+(\nu)$, 
defined for $\nu$ small enough, with the following properties

\begin{itemize}
    \item[i)] $|\alpha| < \alpha_+(\nu)$, close to $\alpha_+(\nu)$, 
    there exists one and only one unstable eigenvalue $\lambda(\alpha,\nu)$, which moreover is simple,
    with corresponding eigenvector $\zeta(x,y)$ of the form
$$
\zeta(x,y) = \nabla^{\perp} \Bigl[ e^{i \alpha x} \psi(y) \Bigr],
$$
for some stream function $\psi(y)$ depending on $\alpha$ and $\nu$.

    \item[ii)] for $|\alpha| > \alpha_+(\nu)$, 
    $$
    \sup_{\lambda \in Sp_{\alpha,\nu}} \Re \lambda < 0 .
    $$

    \item[iii)] for $| \alpha | = \alpha_+(\nu)$, there exists only one eigenvalue
    with a zero real part, $\lambda(\alpha,\nu)$. This eigenvalue is simple
    and moreover
    \beq \label{derivative}
    (\partial_\nu \Im \lambda) (\alpha_+(\nu),\nu) > 0.
    \eeq
\end{itemize}
}

This assumption is proved for convex or concave shear flows in \cite{BG4}. 
Note that these points may be easily  numerically investigated.

Note that if $\lambda$ is an eigenvalue corresponding to a wavenumber $\alpha$,
$\bar \lambda$ is an eigenvalue corresponding to the eigenvalue $-\alpha$.


\subsection{Statement of the result}


Let us now state our main result. 

\begin{theorem}\label{t:red}
Let us assume that $U_s$ is a concave or convex analytic function, 
or more generally that $U_s$ satisfies the assumption (H). 
Let $\nu_\star > 0$ be small enough and set $\alpha_+=\alpha_+(\nu_\star)$. Let $\zeta(x,y)$ be
the eigenvector corresponding to the {purely imaginary eigenvalue} $\lambda(\alpha_+,\nu_\star)$ and let
$$
\mu = \nu - \nu_\star.
$$ 
Then, there exists a real number $\omega_+$ and complex numbers $c_1$ and $c_3$ such that for any $\mu$ sufficiently small, 
any sufficiently small bounded solution of \eqref{NS1}-\eqref{NS3} that is $2\pi/\alpha_+$-periodic in $x$ is of the form
\beq \label{formu}
u^\nu(t,x,y) = U(y)+ A(t) \zeta(x,y) + \bar A(t) \bar \zeta(x,y) + O \Bigl( |A| (| \mu | + |A|) \Bigr),
\eeq
where the complex-valued function $A$ satisfies the ordinary differential equation
\beq \label{formu2}
{dA \over dt} = i \omega_+ A + c_1 \mu A + c_3 A |A|^2+O \Bigl( |A| (|\mu|^2+|A|^4) \Bigr).
\eeq
Furthermore,
\[
i\omega_+ = \lambda(\alpha_+,\nu),\qquad c_1 = \partial_\alpha \lambda(\alpha_+,\nu),
\]
and $\Re c_1 < 0$.
If moreover $\Re c_3 < 0$, the Hopf bifurcation is supercritical. If $\Re c_3 > 0$, it is subcritical.
\end{theorem}

As usual in bifurcation theory, any small solution of (\ref{formu2}) provides a solution of \eqref{NS1}-\eqref{NS3} of the form (\ref{formu}).
The proof of this result relies on a center manifold reduction.

There is no theoretical result on the coefficient $c_3$ and in particular on the sign of $\Re c_3$.
However, the value of $c_3$ can be {\it numerically} evaluated \cite{BG5}. 

If this Hopf bifurcation is supercritical, a stable time-periodic solution bifurcates for $\mu < 0$ small enough, of the form 
\begin{equation}\label{e:roll}
u^\nu(t,x,y) =  U(y)+ |\mu|^{1/2}  \Re \nabla^\perp \Bigl[  \sqrt{c_{1r} \over c_{3r}} \, e^{ i \omega_+ t + i \alpha_+ x} \psi(y) \Bigr] + O(|\mu|),
\end{equation}
where $c_{1r} = \Re c_1$ and $c_{3r}=\Re c_3$.
In particular,  the difference $u^\nu-U(y)$ is of size $O(\mu^{1/2})$ in $L^\infty$, and  depends on 
$y$ and on the combination $i \omega_+ t + i \alpha_+ x$.  It thus takes the form of a non-stationary ``roll'' in $(x,y)$
which moves with a constant speed $- \omega_+/\alpha_+$.
This justifies the second step of the spectral ``route to chaos":  the exponentially growing perturbations get saturated by
nonlinear effects and converge to a new ``base state", composed of moving rolls.

If this Hopf bifurcation is subcritical, there exists a particular solution of Navier-Stokes equation which is time and space periodic,
but this solution is unstable.

\medskip

Up to the best of our knowledge, this result is the first bifurcation result for the genuine Navier-Stokes equations
on a strip. In particular, it opens the way to the study of secondary instabilities.
We believe that the assumption (H) is generically true, but the proof of (H) deserves further investigations.


\subsection{Estimates on the resolvent \label{estimates}}


The Green function of Orr-Sommerfeld equation has been constructed and studied in  details in \cite{BG4} as $\nu$ goes to $0$. 
In this paper however, $\alpha$, $c$ and $\nu$ are small but fixed, and thus can be considered as being of order $O(1)$.
The inversion of (\ref{Orrmod0}) is then much easier.

\begin{proposition}
Let $\alpha_0 > 0$, Let $\psi$ and $f$ be such that 
$$
OS_{\lambda,\alpha,\nu}(\psi) = f.
$$
Then, uniformly in $| \alpha | \ge \alpha_0$ and in $\lambda$ large enough, we have
\beq \label{bound0}
\|  \psi \|_{L^2} + \|  \partial_y \psi \|_{L^2}  + \|  \partial_y^2 \psi \|_{L^2}   \lesssim | \lambda |^{-1} \| f \|_{L^2}
\eeq
where $L^2$ denotes $L^2(0,1)$.
The estimate (\ref{bound0}) is also true when $\alpha = 0$.
\end{proposition}

\begin{proof}
When $\alpha = 0$, the Orr-Sommerfeld equation reduces to
$$
- \lambda \partial_y^2 \psi + \nu \partial_y^4 \psi = f,
$$
thus (\ref{bound0}) is true.

When $\alpha \ne 0$, we rewrite Orr-Sommerfeld equations as 
\beq \label{eq1}
- \eps \partial_y^4 \psi - \tilde c  \partial_y^2 \psi + \alpha^2 c \psi = E(\psi)
\eeq
where
\beq \label{eq2}
E(\psi) = - U_s \partial_y^2 \psi  
+ \alpha^2 U_s  \psi + U_s'' \psi + \eps \alpha^4 \psi  + f,
\eeq
where $f$ is some forcing term and where 
$$
\tilde c = c - 2 \eps  \alpha^2.
$$
We recall that $\lambda = - i \alpha c$ and that $\alpha$ is fixed, thus large $\lambda$ corresponds to large $c$ and $\tilde c$.
We first consider the approximate equation
\beq \label{app}
- \eps \partial_y^4 \psi - \tilde c \partial_y^2 \psi + \alpha^2 c \psi = f .
\eeq
Looking for solutions of the form $e^{\mu y}$ leads to the characteristic equation 
$$
- \eps \mu^4 - \tilde c \mu^2 + \alpha^2 c = 0,
$$
and thus to
$$
\mu^2 = { - \tilde c  \pm \sqrt{\tilde c^2 + 4 \eps \alpha^2 c} \over 2 \eps } .
$$
There exists thus two ``slow solutions'' $e^{\pm \mu_s y}$ where $\mu_s$ is of order $O(1)$ and
two ``fast solutions'' $e^{\pm \mu_f y}$ where $\mu_f \sim \sqrt{\tilde c / \eps}$, namely large when $\lambda$ is large.

The Green function $G_0(x,y)$ of (\ref{app}), namely the solution corresponding to $f = \delta_x$, is the sum of an ``interior" Green function $G^{int}(x,y)$ which
takes care of the source $f$, and of a "boundary layer" Green function $G^b(x,y)$ which recovers the boundary condition. 

In order to ensure the symmetry with respect to $y = 1/2$, we choose $G^{int}(x,y)$ of the form
$$
G^{int}(x,y) = a  \Bigl[ e^{- \mu_f | x - y |}  + e^{- \mu_f | x -(1 - y) |}  \Bigr] + b  \Bigl[ e^{- \mu_s |x - y| } + e^{- \mu_s |x - (1 - y) | }  \Bigr]
$$
for some constants $a$ and $b$ to be determined.
Note that, by construction, $G^{int}(x,y)$ and $\partial_y^2 G^{int}(x,y)$ are continuous at $x$.
We must write that $\partial_y G^{int}(x,y)$ is also continuous, and that $\partial_y^3 G^{int}(x,y)$ has a jump of size $-1 / \eps$,
which leads to
$$
\mu_f a + \mu_s b = 0
$$
and
$$
\mu_f^3 a + \mu_s^3 b = {1 \over 2 \eps} .
$$
We thus have
$$
a = {1 \over 2 \eps} {1 \over \mu_f} {1 \over \mu_f^2 - \mu_s^2} = O(  | \lambda |^{-3/2})
$$
and
$$
b = - {1 \over 2 \eps} {1 \over \mu_s} {1 \over \mu_f^2 - \mu_s^2} = O( | \lambda |^{-1}).
$$
In particular,  for large $\lambda$,
\beq \label{estimGint}
\int_0^1 | G^{int}(x,y) |  \, dy + \int_0^1 | \partial_y G^{int}(x,y) |  \, dy + \int_0^1 | \partial_y^2 G^{int}(x,y) | \, dy 
\lesssim | \lambda |^{-1} 
\eeq
and similarly for integrals over $x$.

We now look for $G^b(x,y)$ under the form
$$
G^b(x,y) = a' \Bigl[ e^{- \mu_f y} + e^{- \mu_f (1 - y)} \Bigr] + b' \Bigl[ e^{- \mu_s y} + e^{- \mu_s (1 - y)} \Bigr]
$$
where the coefficients $a'$ and $b'$ depend on $x$.
We then write
$$
G^{int}(x,0) + G^b(x,0) = \partial_y G^{int}(x,0) + \partial_y G^b(x,0) = 0,
$$
which explicitly gives $a'$ and $b'$. It can then be checked that $G^b(x,y)$ also satisfies estimates of the form (\ref{estimGint}).

We now solve the genuine Orr-Sommerfeld equation, written under the form (\ref{eq2}) by iteration, starting from 
$$
\psi_1 = G_0 \star_x f 
$$
where $G_0 = G^{int} + G^b$.
Note that $\| \psi_1 \|_{L^2} \lesssim | \lambda |^{-1} \| f \|_{L^2}$.
We then define by induction $E_n = E(\psi_n)$ and $\psi_{n+1} = G_0 \star_x E_n$. 
As $\| E_n \|_{L^2} \lesssim \| \psi_n \|_{L^2}$, this procedure converges, and the solution of (\ref{eq2}) is given by
$$
\psi = \sum_j \psi_j ,
$$
and satisfies (\ref{bound0}).
\end{proof}


\section{Bifurcation analysis \label{part3}}


We are interested in solutions which bifurcate close to $(\alpha_+(\nu),\nu)$. 
We fix $\nu_\star$ small enough, so that the assumption (H) holds, 
set $\alpha_+=\alpha_+(\nu_\star)$ and take $\mu = \nu - \nu_\star$ as bifurcation parameter.


\subsection{Dynamical system formulation}


For the system (\ref{NS1})-(\ref{NS3}), we consider solutions of the form
$$
u^\nu(t,x,y) = U(y) + v(t,x,y),
$$
in which $U(y)$ is the shear profile.
The perturbation $v$ satisfies the system
\beq \label{NSb1}
\partial_t v = - L_{NS} \, v -  (v \cdot \nabla) v,
\eeq
\beq \label{NSb2}
\nabla \cdot v = 0,
\eeq
\beq \label{NSb3}
v = 0, \quad \hbox{for} \quad y = 0 \quad \hbox{and} \quad y = 1,
\eeq
in which $L_{NS}$ is the operator defined by \eqref{linearNS}.
For this system, we look for solutions $v$ which are $2 \pi/\alpha_+$-periodic in $x$.

As a first step, we formulate the
system \eqref{NSb1}-\eqref{NSb3} as a dynamical system of the form
\beq \label{1}
\partial_t v = L_\nu v + B(v) ,
\eeq
in a suitably chosen phase space ${\cal X}$, in which
$L_\nu$ is a closed operator with domain ${\cal Z}$, the phase space ${\cal X}$ includes the condition $\nabla \cdot v = 0$ and the domain ${\cal Z}$ includes
the remaining boundary conditions, and, more importantly, the gradient term $\nabla q$ is eliminated by a suitable projection.
Such a formulation is well-known (see for instance \cite{CI} or \cite[Chapter 5]{Haragus}).
We  take
$$
{\cal X} = \Bigl\{v\in (L^2_\per(\Omega))^2 \quad | \quad \nabla\cdot v=0,\quad v_2|_{ y = 0} =v_2|_{ y = 1} = 0 \Bigr\}, 
$$
where  $\Omega = \rit \times (0,1)$,  the subscript $\per$ means that the
functions are $2\pi/\alpha_+$-periodic in $x$ and where $v=(v_1,v_2)$. 
Next, we consider the Leray projection $\Pi_0: (L^2_\per(\Omega))^2 \to {\cal X}$, defined as the orthogonal projection on
${\cal X}$ with respect to the standard $(L^2_\per(\Omega))^2$ scalar product, so that
$$
\Pi_0 \nabla q = 0.
$$
Then, projecting with $\Pi_0$, we obtain the dynamical system formulation (\ref{1}) in which
$$
L_\nu = - \Pi_0 L_{NS},
\qquad B(v) = - \Pi_0 (v \cdot \nabla v) .
$$ 
In this formulation, $L_\nu$ is a closed linear operator in ${\cal X}$ with domain
$$
{\cal Z} = \Bigl\{ v \in (H^2_\per(\Omega))^2 \quad | \quad \nabla \cdot v = 0, \quad v|_{y = 0} =v|_{y = 1} = 0 \Bigr\} .
$$
Moreover, $B: {\cal Z} \to {\cal X}$ is well-defined and quadratic, thus $B$ is of class $C^k$ for any $k \ge 1$.


\subsection{Spectral analysis}


The following properties of the linear operator $L_{\nu}$ are an immediate consequence of the results of Section~\ref{estimates} and of the assumption (H). 
\begin{lemma}\label{l:linear}
\begin{itemize}
\item[(i)]
The linear operator  $L_{\nu_\star}$ has precisely two simple purely imaginary eigenvalues $\pm i \omega_+$ with associated eigenvectors $\zeta$  and $\bar \zeta$ of the form
$$
\zeta(x,y) = e^{i \alpha_+ x} \phi(y).
$$
Furthermore, there exists $\sigma > 0$ such that the spectrum $\sigma(L_{\nu_\star})$ satisfies
$$
\sigma(L_{\nu_\star}) \setminus \{\pm i \omega_+ \} \subset \{ \lambda \in \cit \quad | \quad \Re \lambda < - \sigma \} .
$$
\item[(ii)] There exist constants $C > 0$ and $\omega_0 > 0$ such that the following estimate holds
$$
\| ( L_{\nu_\star} - i \omega)^{-1} \|_{ {\cal X} \to {\cal X}} \le {C \over | \omega |},
$$
for any $| \omega | \ge \omega_0$. 
\end{itemize}
\end{lemma}


\subsection{Center manifold reduction}


Lemma \ref{l:linear} allows us to apply a center manifold theorem to the dynamical system \eqref{1} for $\nu$
close to $\nu_\star$ and for $\alpha_+=\alpha_+(\nu_\star)$. This will prove Theorem~\ref{t:red}.

We set $\nu =\nu_\star+ \mu$, so that $\mu$ is now a small bifurcation parameter.
The properties in Lemma~\ref{l:linear} and the property that $B$ is quadratic imply that the hypotheses of the center manifold theorem 
\cite[Chapter $2$: Theorems $2.20$ and $3.3$]{Haragus} are satisfied, for $\mu$ sufficiently small. 
As a consequence, small bounded solutions of \eqref{1} are of the form
$$
v(t) = A(t) \zeta + \bar A(t) \bar \zeta + O \Bigl( |A|(| \mu | + |A|) \Bigr), 
$$
in which $\zeta$ and $\bar\zeta$ are the eigenvectors of $L_{\nu_\star}$ in Lemma~\ref{l:linear} and  $A$ is a complex-valued function  satisfying an ordinary differential equation of the form
\beq \label{3}
{d A \over d t} = i \omega_+ A_+ + f( \mu, A, \bar A).
\eeq
Note that $\bar A(t)$ satisfies the complex conjugated equation. 

We can obtain a more precise description of $f(\mu,A,\bar A)$ by taking into account the invariance of the equations \eqref{NS1}-\eqref{NS3} under translations in $x$. Indeed, this invariance implies that the dynamical system (\ref{1}) is equivariant under the action of the continuous group $(\tau_a)_{a \in \rit/2\pi\zit}$ defined by
$$
\tau_a v(x,y) = v(x+a/\alpha,y).
$$
This symmetry is inherited by the reduced system (\ref{3}). Since
$\tau_a \zeta = e^{i a } \zeta$,
the action of $\tau_a$ on $A(t)$ reads 
$$
\tau_a A = e^{ia} A, \qquad \forall a \in \rit/2\pi\zit.
$$
Consequently,  the equivariance of \eqref{3} under the action of $\tau_a$ implies that 
$$
f(e^{i a} A , e^{- i a } \bar A,\mu) = e^{i a} f(A,\bar A,\mu),
$$
so that
\begin{equation}\label{fg}
f(A,\bar A,\mu) = A g(|A|^2,\mu),
\end{equation}
according to \cite[ Lemma $2.4$, Chapter $1$]{Haragus}. To leading order, we find the equation
\beq \label{4}
{d A \over dt} = i \omega_+ A + c_1 \mu A + c_3 A |A|^2,
\eeq
the higher order terms being $O(|A|(|\mu|^2+|A|^4))$.
This proves the result in Theorem~\ref{t:red} (in which $\omega_+$, $c_1$ and $c_3$ are the ones above multiplied by $\alpha_+$).

We note that 
the symmetry of $U_s(y)$ with respect to $y=1/2$, i.e., $U_s(y)=U_s(1-y)$, implies the equivariance of \eqref{3} under the reflection $S$ defined by
\[
S(v_1,v_2)(x,y) = (v_1,-v_2)(x,1-y),
\]
and the equivariance of the reduced system under the induced reflection.
Since $S\zeta=-\zeta$,
the action of the induced reflection on $A(t)$ is given by $SA=-A$. This implies that the reduced vector field $f$ is odd in $(A,\bar A)$, a property that is clearly satisfied by the vector field $f$ from \eqref{fg}.


\subsection{Hopf bifurcation}


To analyze the solutions of \eqref{4}, we need to compute the signs of the real parts of the coefficients $c_1$ and $c_3$.

The coefficient $c_1$ can be related to the eigenvalue $\lambda(\alpha_+,\nu_\star)$. 
Indeed, $\lambda(\alpha_+,\nu_\star)$ is precisely the eigenvalue
of the linearization of the reduced equation (\ref{3}) at $A = 0$. This implies that
$$
\lambda(\alpha_+,\nu_\star + \mu) =  i \omega_+ + c_1 \mu + O(\mu^2) ,
$$
which then allows to determine the sign of $\Re c_1$. The assumption (H) directly gives $\Re c_1 < 0$.

To compute $c_3$, one can use the method given in \cite[Chapter $1$, page $16$]{Haragus} which provides an analytical formula of $c_3$. In the present case this type of computation seems out of reach, because it requires an explicit inversion of the Orr-Sommerfeld equation.
However, it is possible to evaluate $c_3$ numerically \cite{BG5}.
In the case of a symmetric profile, these numerical computations show that $\Re c_3 < 0$.

The signs of the coefficients $c_1$ and $c_3$ found above show that for the  differential equation \eqref{4},
a supercritical Hopf bifurcation occurs at $\mu=0$, and the same property holds for the reduced equation \eqref{3} (e.g., see \cite[Theorem 2.6, Chapter 1]{Haragus}). More precisely, we have the following properties:
\begin{itemize}
\item[(i)] for any sufficiently small $\mu$, the differential equation \eqref{3} possesses the equilibrium $A=0$ which is stable if $\mu>0$ and unstable if $\mu<0$;
\item[(ii)] for any sufficiently small $\mu<0$, the differential equation \eqref{3} possesses a unique periodic orbit $A_\mu=O(|\mu|^{1/2})$ which is stable.
\end{itemize}
Furthermore, setting $c_{1r}=\Re c_1$ and  $c_{3r}=\Re c_3$, we have that
\[
|A_\mu(t)| = \sqrt{\frac{c_{1r}}{c_{3r}}}|\mu|^{1/2} + O(|\mu|).
\]
Going back to the dynamical system \eqref{1} we obtain the stable periodic solution
\[
v(t)=  2\Re \Bigl(  A_\mu(t) \zeta \Bigr) + O(|\mu|) 
= 2 |\mu|^{1/2} \sqrt{\frac{c_{1r}}{c_{3r}}}\, \Re\Bigl( e^{i\omega_+t+i\alpha_+ x} \phi(y) \Bigr) + O(|\mu|),
\]
for $\mu<0$ sufficiently small, and going back to the original system \eqref{NS1}-\eqref{NS3} we obtain the solution \eqref{e:roll}.


\subsubsection*{Acknowledgements}


The authors would like to thank G. Iooss for fruitful discussions.
D. Bian is supported by NSFC under the contract 12271032.

\subsubsection*{Conflict of interest}  The authors state that there is no conflict of interest.

\subsubsection*{Data availability}  Data are not involved in this research paper.





\begin{thebibliography}{99}



\bibitem{BG4} D. Bian, E. Grenier: Spectrum of Orr-Sommerfeld in a strip, {\it preprint}, 2025.

\bibitem{BG5} D. Bian, E. Grenier:  Onset of nonlinear instabilities in monotonic viscous boundary layers, {\it SIAM
J. Math. Anal.} 56(3), 3703-3719, 2024.

\bibitem{BGI} D. Bian, E. Grenier, G. Iooss: Bifurcations of viscous boundary layers in the half space, {\it preprint}, $2025$. 


\bibitem{CI} P. Chossat, G. Iooss: The Couette-Taylor problem.  {\it Applied Mathematical Sciences}, 102. Springer-Verlag, New York, 1994.
  
\bibitem{Reid}  P. G.  Drazin, W. H. Reid:  Hydrodynamic stability, {\it  Cambridge Monographs on Mechanics and Applied Mathematics}. Cambridge University, Cambridge--New York, 1981.

\bibitem{Guo2}  E. Grenier, Y. Guo, and T. Nguyen: Spectral instability of characteristic boundary layer flows, {\it Duke Math. J., } 165(16), 3085--3146, 2016.

 \bibitem{GN2} E. Grenier, T. Nguyen: $L^\infty$ instability of Prandtl layers, {\it Ann. PDE}, 5(2), 2019.
 
\bibitem{Haragus} M. Haragus, G. Iooss: Local bifurcations, center manifolds, and normal forms in infinite dimensional dynamical systems, {\it Springer}, 2011.

\bibitem{Reid2} W. H. Reid: The stability of parallel flows, {\it Developments in Fluid dynamics}, Vol $1$, Academic Press, 1965.

\bibitem{Schmidt} P.J. Schmid, D.S. Henningson: Stability and transition in shear flows,  {\it Applied Mathematical Sciences}, 142,  Springer-Verlag, New York,  2001.




\end{thebibliography}
\end{document}